\documentclass{amsart}

\usepackage{color,graphicx,amssymb,latexsym,amsfonts,txfonts,amsmath,amsthm}
\usepackage{pdfsync}
\usepackage{amsmath,amscd}
\usepackage[all,cmtip]{xy}

\usepackage{hyperref}
\hypersetup{
    colorlinks=true,       
    linkcolor=blue,          
    citecolor=blue,        
    filecolor=blue,      
    urlcolor=blue           
}

\input epsf

\newtheorem{theo}{Theorem}[section]
\newtheorem{coro}[theo]{Corollary}

\newtheorem{lemm}[theo]{Lemma}

\theoremstyle{remark}
\newtheorem{rema}[theo]{\bf Remark}

\begin{document}

\title{Regular dessins d'enfants with dicyclic group of automorphisms}
\author{Rub\'en A. Hidalgo}

\author{Sa\'ul Quispe}


\address{Departamento de Matem\'atica y Estad\'{\i}stica, Universidad de La Frontera. Temuco, Chile}
\email{ruben.hidalgo@ufrontera.cl, saul.quispe@ufrontera.cl}
\thanks{Partially supported by Projects Fondecyt 1190001, 11170129 and Anillo ACT1415 PIA-CONICYT}

\begin{abstract}
Let $G_{n}$ be the dicyclic group of order $4n$. We observe that, up to isomorphisms, (i) for $n \geq 2$ even there is exactly one regular dessin d'enfant with automorphism group $G_{n}$, and (ii) for $n \geq 3$ odd there  are exactly two of them. All of them are produced on very well known hyperelliptic Riemann surfaces. We observe, for each of these cases, that the isotypical decomposition, induced by the action of $G_{n}$, of its jacobian variety has only one component. If $n$ is even, then the action is purely-non-free, that is, every element acts with fixed points. In the case $n$ odd, the action is not purely-non-free in one of the actions and purely non-free for the other.
\end{abstract}

\maketitle

\section{Introduction}
Is a consequence of the Riemann-Roch's theorem, in one direction, and the Implicit Function theorem, in the other, that there is an equivalence between the category of closed Riemann surfaces (and holomorphic maps as morphisms) on one side and that of irreducible smooth complex algebraic curves (and rational maps as morphisms) on the other. Belyi's theorem \cite{Belyi} asserts that a closed Riemann surface $S$ can be described by an algebraic curve defined over the field $\overline{\mathbb Q}$ of algebraic numbers if and only if there is a non-constant meromorphic map $\beta:S \to \widehat{\mathbb C}$ whose branch values are contained in the set $\{\infty,0,1\}$. In this case, $\beta$ is called a {\it Belyi map}, $S$ a {\it Belyi curve} and $(S,\beta)$ a {\it Belyi pair}. 
To each Belyi pair $(S,\beta)$ there is associated an embedding of a bipartite graph ${\mathcal G}_{\beta}$ in $S$ (the fibres $\beta^{-1}(0)$ and $\beta^{-1}(1)$ providing, respectively, the white and black vertices, and its edges being $\beta^{-1}([0,1])$) such that $S\setminus {\mathcal G}_{\beta}$ consists of topological discs. 

A pair $(X,{\mathcal G})$, where $X$ is a closed orientable surface and ${\mathcal G} \subset X$ is a bipartite graph such that the connected components of $X\setminus {\mathcal G}$ are topological discs, is called a {\it dessin d'enfant}, as introduced by Grothendieck~\cite{Gro} (see also, the recent books \cite{GiGo,JW}). In particular, if $(S,\beta)$ is a Belyi pair, then $(S,{\mathcal G}_{\beta})$ is a dessin d'enfant. Conversely, if $(X,{\mathcal G})$ is a dessin d'enfant, then the uniformization theorem asserts the existence of a (unique up to isomorphisms) Belyi pair $(S,\beta)$ and of an orientation-preserving homeomorphism $f:S \to X$ such that $f({\mathcal G}_{\beta})={\mathcal G}$ ($f$ sends black vertices to black vertices). This process provides of a natural equivalence between (isomorphism classes of) Belyi pairs and (isomorphism classes of) dessins d'enfants and, by Belyi's theorem,  there is a natural action of the absolute Galois group ${\rm Gal}(\overline{\mathbb Q}/{\mathbb Q})$ on (isomorphism classes of) dessins d'enfants. It is well known that this action is faithfull \cite{GiGo1,Gro,Schneps}. Associated to a Belyi pair $(S,\beta)$ (respectively, a dessin d'enfant ${\mathcal D}=(X,{\mathcal G})$) is its group ${\rm Aut}^{+}(S,\beta)$ (respectively, ${\rm Aut}^{+}({\mathcal D})$) consisting of those conformal automorphisms $\varphi$ of $S$ such that $\beta=\beta \circ \varphi$ (respectively, the group of homotopy class of those orientation-preserving self-homeomorphisms of $X$ keeping invariant ${\mathcal G}_{\beta}$ and the colour of the vertices). The Belyi pair (respectively, dessin d'enfant) is called {\it regular} if $\beta$ is a regular branched cover with deck group ${\rm Aut}^{+}(S,\beta)$ (respectively, the group ${\rm Aut}^{+}({\mathcal D})$ acts transitively on the set of edges of ${\mathcal G}$). In \cite{BCG,GoJa} it was observed that the absolute Galois group also acts faithfully at the level of regular dessins d'enfants.  

A finite group $G$ appears as the group of automorphisms of a regular dessin d'enfant if and only if there is a closed Riemann surface $S$ over which $G$ acts as a group of conformal automorphisms such that $S/G$ has genus zero and exactly three cone points; we say that the conformal $G$-action on $S$ is {\it triangular}. This is equivalent for $G$ to be generated by two elements. There is a bijection between (i) equivalence classes of regular dessins d'enfants with automorphism group $G$ and (ii) $G$-conjugacy classes of pairs of generators of $G$.
Examples of groups generated by two elements are the {\it dicyclic groups}
\begin{equation}\label{eqq1}
G_{n}=\langle x,y: x^{2n}=1, y^{2}=x^{n}, yxy^{-1}=x^{-1}\rangle, \; n \geq 2.
\end{equation}

In \cite{MZ7} it was proved that the strong symmetric genus $\sigma^{0}(G_{n})$ (the lowest genus over which $G_{n}$ acts as a group of conformal automorphisms) is equal to (i) $n$, for $n$ even and (ii) $n-1$, for $n$ odd. For $n$ even, such a minimal conformal action happens for $S/G_{n}$ of signature $(0;4,4,2n)$ and, for $n$ odd, it happens for signature $(0;4,4,n)$. In Theorem \ref{triangular} we prove that, up to isomorphisms, there is only one triangular action of $G_{n}$ when $n \geq 2$ is even, and that there are exactly two when $n \geq 3$ is odd. In any of these cases, the triangular action is produced in a familiar kind of hyperelliptic Riemann surfaces (for example, if $n$ is odd, then these are the Accola-Maclachlan surfaces \cite{Maclachlan}).
In Remark \ref{grafo} there are described the corresponding monodromy groups and the associated bipartite graph.

A conformal action of a finite group $G$ over a closed Riemann surface $S$ is called {\it purely-non-free} if every element acts with a non-empty set of fixed points. This type of conformal actions where discussed by J. Gilman in her study of adapted basis for the derived action on the first integral homology group \cite{gilman}. In \cite{GH} it was observed that every finite group $G$ acts purely-non-free on some Riemann surface. We define the {\it pure symmetric genus}  $\sigma_{p}(G)$ as  the minimal genus  on which $G$ acts purely-non-free. For $G$ either cyclic, dihedral, ${\mathcal A}_{4}$, ${\mathcal A}_{5}$ and ${\mathfrak S}_{4}$ it holds that $\sigma_{p}(G)=\sigma^{0}(G)=0$. In \cite{GH} it was computed $\sigma_{p}(G)$ for $G$ an abelian group (in general, $\sigma^{0}(G)<\sigma_{p}(G)$). In Theorem \ref{triangular} we observe that the triangular action of $G_{n}$ on a closed Riemann surface of genus $n$ is purely-non-free and that on genus $n-1$ is not, in particular, $\sigma_{p}(G_{n})=n$ (see Corollary \ref{coro1}).

A conformal action of a finite group $G$ on a Riemann surface $S$ of genus $g \geq  2$ induces a natural ${\mathbb Q}$-algebra homomorphism $\rho : {\mathbb Q}[G]\to {\rm End}_{\mathbb Q}(JS)$, from the group algebra ${\mathbb Q}[G]$ into the endomorphism algebra of the jacobian variety $JS$. The factorization of ${\mathbb Q}[G]$ into a product of simple algebras yields a decomposition of $JS$ into abelian subvarieties, called the {\it isotypical decomposition} (see, for instance, \cite{CR,LR}). In the case of the conformal triangular actions of $G_{n}$, we note in Theorem \ref{triangular} that, for every non-trivial subgroup $H$ of $G_{n}$, the quotient orbifold $S/H$ has genus zero. In particular, this asserts that the isotypical decomposition of $JS$, induced by the action of $G_{n}$, is trivial (that is, there is only one factor) and the  induced action of any non-trivial subgroup of $G_{n}$ on the jacobian variety $JS$ acts with isolated fixed points (see  Corollary \ref{isolado}).

As $G_{n}$ has index two subgroups, it can be realized as a group of conformal/anticonformal automorphisms of suitable Riemann surfaces, such that it admits anticonformal ones. It is well known that $G_{n}$ acts with such a property in genus $\sigma(G_{n})=1$ \cite{MZ7} (this is the symmetric genus of $G_{n}$ \cite{Burn,Hurwitz,T}). This can be easily seen by considering the Euclidean planar group $\Delta=\langle a(z)=i(1-\overline{z}), b(z)=-i\overline{z}+2+i\rangle=\langle a,b: a^{2}b^{2}=1\rangle$, acting on the complex plane ${\mathbb C}$, and consider the surjective homomorphism $\theta:\Delta \to G_{n}$, defined by $\theta(a)=y$ and $\theta(b)=yx$. If $\Gamma$ is the kernel of $\theta$, then ${\mathbb C}/\Gamma$ is a genus one Riemann surface admitting the group $G_{2n}$ as a group of conformal/anticonformal automorphisms (the element $x$ acts conformally and $y$ anticonformally).  In Theorem \ref{hyper} we observe that the minimal genus $\sigma^{hy}(G_{n})\geq 2$ over which $G_{n}$ acts as a group of conformal/anticonformal automorphisms, and admitting anticonformal ones, is $n+1$ for $n \geq 2$ even, and $2n-2$ for $n \geq 3$ odd.

Moduli space of closed Riemann surface of genus $g \geq 2$, a complex orbifold of dimension $3(g-1)$, has a natural real structure (this coming from complex conjugation). The fixed points of such a real structure, its real points, are the (isomorphism classes of) closed Riemann surfaces admitting anticonformal automorphism. Inside such real locus are those surfaces which cannot be defined by algebraic curves over the reals (this is equivalent to say that they do not have anticonformal involutions); these are called {\it pseudo-real} Riemann surfaces. 
In general, a finite group might not be realized asthe group of conformal/anticonformal automorphisms, admitting anticonformal ones, of a pseudo-real Riemann surface; for instance, in \cite{BCC} it was observed that a necessary condition for that to happen is for the group to have order a multiple of $4$. 
In Theorem \ref{pseudo}, we construct, for each integer $q \geq 2$, a pseudo-real Riemann surface of genus $g=(l-1)(2n-1)$, where $l=n(2q-1)$, whose full group of conformal/anticonformal automorphisms is $G_{n}$.

\section{Main results}

\subsection{Triangular conformal actions of $G_{n}$}

We proceed to describe, for each integer $n \geq 2$, a very well known families of hyperelliptic Riemann surfaces admitting the dicyclic group $G_{n}$ as a group of conformal automorphisms acting in a triangular way. We also describe their full groups of conformal/anticonformal automorphisms. Then, we proceed to prove that these are the only examples of Riemann surfaces admitting a conformal triangular action of $G_{n}$.
We set $\rho_{m}:=e^{2\pi i/m}$.

\subsubsection{\bf The family $S_{n}$}
For $n \geq 2$, let $S_{n}$ be the hyperelliptic Riemann surface, of genus $n$, defined by 
$$w^{2}=z(z^{2n}-1).$$

The following are conformal automorphisms of $S_{n}$:
$$u(z,w)=(\rho_{2n} z, \rho_{4n} w), \; y(z,w)=(1/z, iw/z^{n+1}).$$

We observe that $u^{4n}=y^{4}=1$ and $u \circ y^{-1}= y \circ u^{-1}$,
and the quotient orbifold $S_{n}/\langle u,y\rangle$ has signature $(0;2,4,4n)$. If $n \geq 3$, this is a maximal signature \cite{Singerman}, so ${\rm Aut}^{+}(S_{n})=\langle u, y\rangle$ in these cases (groups of order $8n$). For $n=2$, the surface $S_{2}$ has the extra conformal automorphismof order three $t(z,w)=(i(1-z)/(1+z),2(1+i)w/(z+1)^3)$, and ${\rm Aut}^{+}(S_{2})=\langle u,y,t\rangle$ (a group of order $48$). The conjugation $\tau(z,w)=(\overline{z},\overline{w})$ is an anticonformal involution of $S_{n}$, acting with fixed points, satisfying that $\tau \circ u \circ \tau=u^{-1}$, $\tau \circ y \circ \tau=y^{-1}$, and ${\rm Aut}(S_{n})=\langle {\rm Aut}^{+}(S_{n}),\tau\rangle$.
If we set $x(z,w)=u^{2}(z,w)=(\rho_{n} z, \rho_{2n} w)$, then $y^{2}(z,w)=x^{n}(z,w)=(z,-w)$, 
$x^{2n}=1$,  $y^{-1}xy=x^{-1}$ and  $G_{n} \cong \langle x, y\rangle$. The quotient $S_{n}/G_{n}$ has signature $(0;4,4,2n)$ and $\pi:S_{n} \to \widehat{\mathbb C}$, defined by $\pi(z,w)=-(z^{n}+1/z^{n}-2)/4$, is a regular branched cover with deck group $\langle x,y\rangle$.

\begin{rema}\label{otro1}
Another (singular) model of $S_{n}$ is given by the affine algebraic curve
$$S_{n}: v^{2n}=u^{n}(u-1)(u+1)^{2n-1}.$$

In this model, $G_{n}$ is generated by 
$x(u,v)=(u,\rho_{2n} v)$ and  $y(u,v)=\left(-u,\frac{v^{2n-1}}{u^{n-1} (u+1)^{2n-2}}\right)$.
\end{rema}

\subsubsection{\bf The family $R_{n}$: Accola-Maclachlan surfaces}
For $n \geq 3$ odd, let $R_{n}$ be the hyperelliptic Riemann surface, of genus $n-1$, defined by 
$$w^{2}=z^{2n}-1.$$

These Riemann surfaces are called Accola-Maclachlan surfaces \cite{Maclachlan} and these are known to be the surfacs with group of conformal automorphisms of order $8g+8=8n$ (see also \cite{BBCGG}). The following are conformal automorphisms of $R_{n}$:
$$u(z,w)=(\rho_{2n} z, w), \; y(z,w)=(1/z, iw/z^{n}).$$

The quotient orbifold $R_{n}/\langle u,y\rangle$ has signature $(0;2,4,2n)$. If $n \geq 5$, this is a maximal signature \cite{Singerman}, so ${\rm Aut}^{+}(R_{n})=\langle u, y\rangle$ in these cases. This also holds for $n=3$. The conjugation $\tau(z,w)=(\overline{z},\overline{w})$ is an anticonformal involution of $R_{n}$, acting with fixed points, satisfying that $\tau \circ u \tau=u^{-1}$, $\tau \circ y \circ \tau=y^{-1}$, and ${\rm Aut}(R_{n})=\langle {\rm Aut}^{+}(R_{n}),\tau\rangle$.
If we set $x(z,w)=y^{2}\circ u^{2}(z,w)=(\rho_{n} z, - w)$, then $y^{2}(z,w)=x^{n}(z,w)=(z,-w)$, 
$x^{2n}=1$,  $y^{-1}xy=x^{-1}$ and  $G_{n} \cong \langle x, y\rangle$. The quotient $R_{n}/G_{n}$ has signature $(0;4,4,n)$ and $\pi:R_{n} \to \widehat{\mathbb C}$, defined by $\pi(z,w)=-(z^{n}+1/z^{n}-2)/4$, is a regular branched cover with deck group $\langle x,y\rangle$.

\begin{rema}\label{otro2}
Another model of $R_{n}$ is given by the affine algebraic curve
$$R_{n}: v^{2n}=u^{n}(u-1)^{2}(u+1)^{2n-2}.$$

In this model, $G_{n}$ is generated by 
$x(u,v)=(u,\rho_{2n} v)$ and $y(u,v)=\left(-u,\frac{\rho_{4n} u(u^{2}-1) }{v}\right)$.
\end{rema}

\begin{rema}[Complex multiplication]
Both families of surfaces $S_{n}$ and $R_{n}$ have the property that their jacobian varieties have complex multiplication \cite[Thm. 2.4.4]{Rohde}
\end{rema}

\subsubsection{\bf Some facts on the dicyclic group $G_{n}$}\label{Gn}
If $G_{n}$ is the dicyclic group, presented as in (\ref{eqq1}), then: (1) it has order $4n$, 
(2) $\langle x \rangle \cong {\mathbb Z}_{2n}$ is a normal subgroup of index two, 
(3)  every element of $G_{n}-\langle x \rangle$ has order $4$ and 
(4) there are exactly $n+3$ conjugacy classes, with representatives given in the following table
\smallskip
\begin{center}
\begin{tabular}{c|c|c|c|c|c|c|c|c}
{\rm Rep.} & 1 & x & $x^{2}$ & $\cdots$  & $x^{n-1}$ & $x^{n}$ & y & xy \\\hline 
{\rm Size} & 1 & 2 & 2 & $\cdots$ & 2 & 1 & $n$ & $n$
\end{tabular}
\end{center}

\subsubsection{\bf On the uniqueness of the triangular action of $G_{n}$}

\begin{theo}\label{triangular}
Let $n \geq 2$ be an integer and let $S$ be a closed Riemann surface admitting the dicyclic group $G_{n}$, presented as in \eqref{eqq1}, as a group of conformal automorphisms and acting in a triangular way, that is, $S/G_{n}$ has triangular signature. Then either one of the following ones hold.

\begin{enumerate}
\item[I.-] $S/G_{n}$ has signature $(0;4,4,2n)$ and $S$ is isomorphic to $S_{n}$. The group $G_{n}$ acts  purely-non-free. The number of fixed points of the automorphisms $x$, $y^{2}=x^{n}$, $y$ and $xy$ are, respectively, $2$, $2+2n$, $2$ and $2$.
If $H$ is any non-trivial subgroup of $G_{n}$, then $S/H$ has genus zero.

\item[II.-] $S/G_{n}$ has signature $(0;4,4,n)$ and $S$ is isomorphic to $R_{n}$, in which case $n \geq 3$ is odd. The only elements of $G_{n}$ acting with fixed points on $R_{n}$ are $x^{2}, x^{4},\ldots, x^{2n-2}, x^{n}$ and those of the form $x^{a}y$. In particular, the action of $G_{n}$ is not purely-non-free.
If $H$ is any non-trivial subgroup of $G_{n}$, then $S/H$ has genus zero.

\end{enumerate}

\end{theo}

\begin{rema}\label{grafo}
Let us consider the following permutations of ${\mathfrak S}_{4n}$
$$
\eta=(1,2,\ldots,2n)(2n+1,2n+2,\ldots,4n),\;
\sigma=\prod_{k=1}^{n}(k,4n+1-k,n+k,3n+1-k).
$$

Then, $\eta^{2n}=1$, $\sigma^{-1}\eta\sigma=\eta^{-1}$, 
$\eta^{n}=\prod_{k=1}^{n}(k,n+k)(2n+k,3n+k)=\sigma^{2}$, 
and $\langle \eta,\sigma\rangle \cong G_{n}$, where the isomorphism is the one taking $\eta$ to $x$ and $\sigma$ to $y$. 
 (1) If $\tau=\sigma^{3}\eta$, then $\tau\sigma\eta=1$, and the pair $(\sigma,\tau)$ determines the monodromy group associated to the regular dessin d'enfant of signature $(0;4,4,2n)$ as described in the above theorem. This permits to see that the associated bipartite graph of this dessin d'enfant is  ${\mathcal C}_{2n}^{2}$, this being defined by taking as its vertices the $2n$-roots of unity (they are painted black an white in alternating way) and we replace each arc of the unit circle joining two consecutive roots (following one of the orientations of the unite circle) by two edges connecting them.
(2) If $n\geq 3$ is odd, and  $\tau=\eta^{n-2}\sigma$, then $\tau\sigma\eta^{2}=1$, where $\eta^{2}=
(1,3,\ldots,2n-1)(2,4,\ldots,2n)(2n+1,2n+3,\ldots,4n-1)(2n+2,2n+4,\ldots,4n)$, and the pair $(\sigma,\tau)$ determines the monodromy of the dessin of signature $(0;4,4,n)$.
\end{rema}

\begin{coro}[On the pure symmetric genus of $G_{n}$]\label{coro1}
$\sigma_{p}(G_{n})=n$.
\end{coro}

\subsubsection{\bf Jacobian variety for the triangular actions of $G_{n}$}
Let $S$ be a closed Riemann surface admitting a triangular conformal action of $G_{n}$. Then there is natural action of $G_{n}$ as group of holomorphic group-automorphisms of the jacobian variety $JS$. If $H$ is a non-trivial subgroup of $G_{n}$, then the previous them asserts that $S/H$ has genus zero, which means that the induced action of $H$ on $JS$ only has isolated fixed point set. As a consequence, we obtain the following.

\begin{coro}\label{isolado}
Let $S$ be a closed Riemann surface admitting a triangular conformal action of $G_{n}$. Then the induced action of $G_{n}$ on the jacobian variety $JS$ only has isolated fixed points and, in particular, the isotypical decomposition of $JS$, induced by the action of $G_{n}$, produces only one factor.
\end{coro}

\subsection{On the hyperbolic genus of $G_{n}$}
Next results describes the minimal genus $\sigma^{hyp}(G_{n}) \geq 2$ such that $G_{n}$ acts as a group of conformal/anticonformal automorphisms and containing anticonformal elements.

\begin{theo}\label{hyper}
If $n \geq 2$ is even, then $\sigma^{hyp}(G_{n})=n+1$. If $n \geq 3$ is odd, then $\sigma^{hyp}(G_{n})=2n-2$.
\end{theo}

\subsection{Pseudo-real actions of $G_{n}$}
Next, we provide pseudo-real Riemann surfaces so that $G_{n}$ is the full group of conformal/anticonformal automorphisms.

\begin{theo}\label{pseudo}
If $n \geq 2$, $q \geq 2$ and $l=n(2q-1)$, then there are pseudo-real Riemann surfaces $S$ of genus $g=(l-1)(2n-1)$
such that ${\rm Aut}(S)=G_{n}$ and ${\rm Aut}^{+}(S)=\langle x \rangle$.
\end{theo}
\section{Proof of Theorem \ref{triangular}}
Assume that $G_{n}$ acts triangular on the closed Riemann surface $S$. Let $\pi:S \to {\mathcal O}=S/\langle x \rangle$ be a branched regular cover map with deck group $\langle x \rangle$. As $\langle x \rangle$ is a normal subgroup and $y^{2} \in \langle x \rangle$, the automorphism $y$ induces a conformal involution $\widehat{y}$ of the orbifold ${\mathcal O}$, so it permutes the branch values of $\pi$ (i.e., the cone points of ${\mathcal O}$) and $S/G_{n}={\mathcal O}/\langle \widehat{y} \rangle$. The triangular property of the action of $G_{n}$ on $S$ (together the Riemann-Hurwitz formula) ensures that ${\mathcal O}$ has genus zero and that its set of cone points are given by a pair of points $p_{1}$, $p_{2}$ (which are permuted by $\widehat{y}$) and probably one or both fixed points of $\widehat{y}$.
By the uniformization theorem, we may identify ${\mathcal O}$ with the Riemann sphere $\widehat{\mathbb C}$. Up to post-composition of $\pi$ with a suitable M\"obius transformation, we may also assume $\widehat{y}(z)=-z$,  $p_{1}=-1$ and $p_{2}=1$.
As the finite groups of M\"obius transformations are either cyclic, dihedral, ${\mathcal A}_{4}$, ${\mathcal A}_{5}$ or ${\mathfrak S}_{4}$, and 
$G_{n}$ is not isomorphic to any of them, the surface $S$ cannot be of genus zero. This (together Riemann-Hurwitz formula) asserts that $\pm 1$ are not the only cone points of ${\mathcal O}$, at least one of the two fixed points of $\widehat{y}$ must also be a cone point (we may assume that $0$ is another cone point). We claim that $\infty$ is also a cone point of ${\mathcal O}$. In fact, if that is not the case, then $\infty$ induces a cone point of order two of $S/G_{n}$ with the property that the $G_{n}$-stabilizer of any point over it on $S$ does not contains a non-trivial power of $x$. But as the only element of $G_{n}$ having order two is $x^{n}$, we obtain a contradiction.
By the above, the cone points of ${\mathcal O}$ are $\pm 1$, $0$ and $\infty$. Moreover, as $0$ and $\infty$ are fixed by the involution $\widehat{y}$, each one of them has points in its preimage on $S$ with $G_{n}$-stabilizer generated by an element of the form $yx^{k}$. In particular, 
they must produce cone points of order $4$ in $S/G_{n}$ and it follows that $\infty$ and $0$ are cone points of order two in ${\mathcal O}$. As $x^{n}$ cannot generate $\langle x \rangle$, the $\langle x \rangle$-stabilizer of any point over $-1$ (and also over $1=\widehat{y}(-1)$) must be stabilized by a non-trivial power $x^{m}$ (where we may assume $m$ to be a divisor of $2n$) such that $\langle x^{m},x^{n}\rangle=\langle x\rangle$. This in particular asserts that $m$ must be relatively prime to $n$, that is, (i) $m=1$ for $n$ even, and (ii) $m \in \{1,2\}$ for $n$ odd. Summarizing the above, we need to consider two situations:
\begin{enumerate}
\item If $n$ is even, then ${\mathcal O}$ has signature $(0;2,2,2n)$ and $S/G_{n}$ has signature $(0;4,4,2n)$. 
\item If $n$ is odd, then either (i) ${\mathcal O}$ has signature $(0;2,2,2n)$ and $S/G_{n}$ has signature $(0,4,4,2n)$ or (ii) ${\mathcal O}$ has signature $(0;2,2,n)$ and $S/G_{n}$ has signature$(0;4,4,n)$.
\end{enumerate}

\subsection{Case $n$ is even}
In this case, $S/G_{n}$ has signature $(0;4,4,2n)$ and (by the Riemann-Hurwitz formula)
$S$ has genus $g=n$. Also, $\pi:S \to \widehat{\mathbb C}$ is a cyclic branched cover, branched at the points $\pm 1$ (with branch order $2n$) and at the points $0, \infty$ (with branch order $2$). In particular, equations for $S$ must be of the form (see, for instance, \cite{BW})
$$S_{n}: v^{2n}=u^{a}(u-1)^{b}(u+1)^{c},$$
where $a,b,c \in \{1,\ldots,2n-1\}$ are such that: (i) $b,c$ are relatively primes to $n$, (ii) $\gcd(a,2n)=n$, (iii) $\gcd(a+b+c,2n)=n$, (iv) $b+c \equiv 0 \mod(2n)$. In this model,  $\pi$ corresponds to $(u,v) \mapsto u$. Condition (ii) asserts that $a=n$ and (iii) follows from this and (iv).
Moreover, we may change the triple $(a,b,c)$ (module $2n$) by $(\alpha a, \alpha b, \alpha c)$, for $\alpha$ integer relatively prime to $2n$, and by permutation of $b$ with $c$. So, by taking $\alpha$ being the inverse of $c$ module $2n$, we may also assume $b=1$ and $c=2n-1$. In this way, we have obtained the uniqueness, up to isomorphisms, of $S$ (this also provides the equations in Remark \ref{otro1}).
By looking at the branched cover $\pi:S \to \widehat{\mathbb C}$, it can be seen that the number of fixed points of $x$ is exactly two and that $y^{2}$ has exactly $2+2n=2+n+n$ fixed points. As the involution $\widehat{y}$ has exactly two fixed points (each one being the projection of the fixed points of $y^{2}$, and the fact that the number of elements in the conjugacy class of $y$ (and of $xy$) is $n$, we may see that $y$ (also $xy$) must have exactly two fixed points. As a consequence (see Section \ref{Gn}), we have obtained that every element of $G_{n}$ acts with fixed points.
Next, let us observe, from the above, that $S/\langle y \rangle$ has signature 
$(\gamma;4,4,2,\stackrel{n}{\ldots},2)$, for some $\gamma \geq 0$. Now, by the Riemann-Hurwitz formula and the fact that $S$ has genus $n$, we obtain that $\gamma=0$. Similarly, $S/\langle xy \rangle$ has genus zero. The orbifold $S/\langle y^{2} \rangle$ must have signature of the form
$(\delta;2,\stackrel{2+2n}{\ldots},2)$. So again from the Riemann-Hurwitz formula we obtain that $\delta=0$. As $S/\langle x \rangle$ has genus zero and $x^{n}=y^{2}$, we may also see that a quotient orbifold $S/\langle x^{a} \rangle$, where $a=1,\ldots,2^{n-1}-1$, must have genus zero and we are done.

\subsection{Case $n$ is odd}
In the case that $m=1$, then we may proceed as similarly as in the previous situation. So, let us now assume $m=2$ (the arguments is also similar, but with some minor changes as we proceed to describe below). In this case, $S/G_{n}$ has signature $(0;4,4,n)$ and (by the Riemann-Hurwitz formula)
$S$ has genus $g=n-1$. Also, $\pi:S \to \widehat{\mathbb C}$ is a cyclic branched cover, branched at the points $\pm 1$ (with branch order $n$) and at the points $0, \infty$ (with branch order $2$). In particular, equations for $S$ must be of the form (see, for instance, \cite{BW})
$$R_{n}: v^{2n}=u^{a}(u-1)^{b}(u+1)^{c},$$
where $a,b,c \in \{1,\ldots,2n-1\}$ are such that: (i) $\gcd(2n,b)=2=\gcd(2n,c)$, (ii) $\gcd(a,2n)=n$, (iii) $\gcd(a+b+c,2n)=n$, (iv) $b+c \equiv 0 \mod(2n)$. In this model,  $\pi$ corresponds to $(u,v) \mapsto u$. Condition (ii) asserts that $a=n$ and (iii) follows from this and (iv).
Moreover, we may change the triple $(a,b,c)$ (module $2n$) by $(\alpha a, \alpha b, \alpha c)$, for $\alpha$ integer relatively prime to $2n$, and by permutation of $b$ with $c$. So, we may also assume $b=2$ and $c=2n-2$.
In this model the generators $x$ and $y$ of $G_{n}$ are given as described in the theorem.
All the above provides the uniqueness, up to isomorphisms (and the equations as in Remark \ref{otro2}).
By looking at the branched cover $\pi:S \to \widehat{\mathbb C}$, it can be seen that the only elements in $\langle x \rangle$ with fixed points are $x^{2}, x^{4}, \ldots, x^{2n-2}$ and $x^{n}=y^{2}$. As the involution $\widehat{y}$ has exactly two fixed points (each one being the projection of the fixed points of $y^{2}$, and the fact that the number of elements in the conjugacy class of $y$ (and of $xy$) is $n$, we may see that $y$ (also $xy$) must have exactly two fixed points. 
Next, let us first observe, from the above, that $S/\langle y \rangle$ has signature 
$(\gamma;4,4,2,\stackrel{n-1}{\ldots},2)$, for some $\gamma \geq 0$. Now, by the Riemann-Hurwitz formula and the fact that $S$ has genus $n-1$, we obtain that $\gamma=0$. Similarly, $S/\langle xy \rangle$ has genus zero. The orbifold $S/\langle y^{2} \rangle$ must have signature of the form
$(\delta;2,\stackrel{2n}{\ldots},2)$. So again from the Riemann-Hurwitz formula we obtain that $\delta=0$. As $S/\langle x \rangle$ has genus zero and $x^{n}=y^{2}$, we may also see that a quotient orbifold $S/\langle x^{a} \rangle$, where $a=1,\ldots,2n-1$, must have genus zero and we are done.

\section{Proof of Theorem \ref{hyper}}\label{hypersymmetric}
 For $n \geq 3$ odd, there is only one index two subgroup of $G_{n}$, this being $H_{1}=\langle x \rangle$.
If $n \geq 2$ is even, then  $G_{n}$ has three index two subgroups, these being 
$H_{1}=\langle x \rangle$, $H_{2}=\langle x^{2},y\rangle$ and  $H_{3}=\langle x^{2},xy\rangle$. In this case, for each $j=1,2,3$, there are closed Riemann surfaces, of genus at least two, for which $G_{n}$ acts as a group of conformal/anticonformal automorphisms with $H_{j}$ acting conformally and the elements of $G_{n}-H_{j}$ acting anticonformally. As, for $n$ even, the automorphism of $G_{n}$
$$\rho:G_{n} \to G_{n}: \;  \rho(x)=x, \; \rho(y)=xy,$$
permutes $H_{2}$ with $H_{3}$, we only need to consider $H_{1}$ and $H_{2}$. 

Let $G_{n}^{+} \in \{H_{1},H_{2}\}$ (in the case $n \geq 3$ odd, $G_{n}^{+}=H_{1}$) and assume that
$G_{n}$ acts as a group of conformal/anticonformal automorphisms of a closed Riemann surface $S,$ of genus at least two, so that $G_{n}^{+}$ acts conformally. As the orders of the cyclic subgroups of $G_{n}$ are divisors of $2n$ or equal to $4$,
the orders of the conical points (if any) of the quotient orbifold ${\mathcal O}=S/G_{n}^{+}$ must have order of such a form. 
The group $G_{n}$ induces an anticonformal involution $\tau$ on the quotient orbifold ${\mathcal O}$  so that ${\mathcal O}/\langle\tau\rangle=S/G_{n}$.

\begin{lemm}
The anticonformal involution $\tau$ acts without fixed points.
\end{lemm}
\begin{proof}
If  $\tau$ has fixed points, then it must have at least a simple loop (i.e. an oval) consisting of fixed points. This means that we may find a lifting of  $\tau$ in $G_{n}-G_{n}^{+}$ having infinitely many fixed points. The only possibility for such a lifting is to have order two, so it must be $x^{n}$. In the case $G_{n}^{+}=H_{1}$, this is not possible as $x$ is conformal. In the case $n \geq 2$ even and $G_{n}^{+}=H_{2}$, we have that $x^{n}$ is a power of $x^{2}$, which is neither possible as $x^{2}$ is conformal. 
\end{proof}

The previous lemma asserts that the number of conical points (if any) of ${\mathcal O}$ is even, say $2r$, and they are permuted in pairs by the involution $\tau$. This, in particular, asserts that ${\mathcal O}$ has signature of the form $(\gamma;m_{1},m_{1}\ldots,m_{r},m_{r})$, where $m_{j} \geq 2$, and that 
$S/G_{n}={\mathcal O}/\langle \tau \rangle$ is a closed hyperbolic non-orientable surface, say a connected sum of $\gamma+1$ real projective planes and having $r$ cone points, say of orders $m_{1},\ldots,m_{r}$. This means that there is a NEC group
$$\Delta=\langle \alpha_{1},\ldots,\alpha_{\gamma+1},\beta_{1},\ldots,\beta_{r}: \alpha_{1}^{2}\alpha_{2}^{2}\cdots\alpha_{\gamma+1}^{2}\beta_{1}\cdots\beta_{r}=\beta_{1}^{m_{1}}=\cdots=\beta_{r}^{m_{r}}=1\rangle$$
where $\alpha_{j}$ is a glide-reflection and $\beta_{j}$ is an elliptic transformation, and there is a surjective homomorphism
$\Theta:\Delta \to G_{n}$
so that $\Theta(\beta_{j}) \in G_{n}^{+}$ and $\Theta(\alpha_{j}) \in G_{n}-G_{n}^{+}$ with torsion-free kernel $\Gamma=\ker(\Theta)$ so that $S={\mathbb H}^{2}/\Gamma$. The genus of $S$, in this case, is
$$g=1+2n(\gamma+r-1-\sum_{j=1}^{r}1/m_{j}).$$

\subsection{Case $G_{n}^{+}=H_{2}$}
In this case, $n \geq 2$ is even, and 
a minimal genus $g \geq 2$ for $S$ is obtained when $\gamma=0$, $r=2$ and $m_{1}=m_{2}=4$, in which case, $g=n+1$. In this case, 
$\Delta=\langle \alpha_{1},\beta_{1},\beta_{2}: \alpha_{1}^{2}\beta_{1}\beta_{2}=\beta_{1}^{4}=\beta_{2}^{4}=1\rangle$
and the homomorphism $\Theta$ can be assumed to be defines as:  
$\Theta(\alpha_{1})=x$,
$\Theta(\beta_{1})=y$ and $\Theta(\beta_{2})=yx^{n-2}$.

\subsection{Case $G_{n}^{+}=H_{1}$}
Again, to obtain a minimal genus $g \geq 2$, we need to assume  
$\gamma=0$, $r=2$, the values $m_{1},m_{2} \in \{2,\ldots,2n\}$ are divisors of $2n$, with $1/m_{1}+1/m_{2}$ maximal such that  
there is a surjective homomorphism, with torsion-free kernel contained in $\Delta^{+}$ (the index two orientation-preserving half of $\Delta$),
$$\Theta:\Delta=\langle \alpha_{1},\beta_{1},\beta_{2}: \alpha_{1}^{2}\beta_{1}\beta_{2}=\beta_{1}^{m_{1}}=\beta_{2}^{m_{2}}=1\rangle \to G_{n},$$
such that $\Theta(\Delta^{+})=H_{1}$. Up to post-composition by an automorphism of $G_{n}$, it can be assumed to satisfy that
$\Theta(\alpha_{1})=y$, and for $j=1,2$, that $\Theta(\beta_{j})=x^{2a_{j}n/m_{1}}$, where $a_{j} \in \{1,\ldots, m_{j}-1\}$,  $\gcd(a_{j},m_{j})=1$,
$\gcd(2n/m_{1},2n/m_{2})=1$ (as $\langle \Theta(\beta_{1}), \Theta(\beta_{2})\rangle=\langle x\rangle$), 
and (as $y^{2}=x^{n}$) also that $2a_{1}n/m_{1}+2a_{2}n/m_{2} \equiv 0 \mod(n)$.
We may write $m_{1}=AB$ and $m_{2}=AC$, where $\gcd(B,C)=1$. As, in this case, $2n=ABL_{1}=ACL_{2}$, it follows that $L_{1}=CL$ and $L_{2}=BL$; so $2n=ABCL$. Now, $1=\gcd(2n/m_{1},2n/m_{2})=\gcd(CL,BL)=L$, so $2n=ABC$, $m_{1}=AB$ and $m_{2}=AC$. As $\gcd(a_{1},AB)=1=\gcd(a_{2},AC)$ and $Ca_{1}+Ba_{2}=2a_{1}n/m_{1}+2a_{2}n/m_{2} \equiv 0 \mod(n)$, we have the following possibilities:
(1) $A=2\widehat{A}$ and $Ca_{1}+Ba_{2}=r \widehat{A}BC$, so $B=C=1$ and $m_{1}=m_{2}=2n$, (2) $B=2 \widehat{B}$ and $Ca_{1}+2\widehat{B}a_{2}=rA\widehat{B}C$, so $C=1$, $B=2$, $m_{1}=2n$ and $m_{2}=n$, and (3) $C=2 \widehat{C}$ and $2 \widehat{C}a_{1}+Ba_{2}=rAB\widehat{C}$, so $B=1$, $C=2$, $m_{1}=n$ and $m_{2}=2n$. If $n \geq 4$ is even, then cases (2) and (3) are not possible, so the maximal value of $1/m_{1}+1/m_{2}$ happens for $m_{1}=m_{2}=2n$, from which we obtain $g=2n-1 \geq n+1$. If $n \geq 5$ is odd, then the maximal value of $1/m_{1}+1/m_{2}$ happens for (up to permutation of indices) $m_{1}=n$ and $m_{2}=2n$, from which we obtain $g=2n-2$.

\section{Proof of Theorem \ref{pseudo}}
Let $l=n (2q-1)$, where $q \geq 2$. Let us consider the NEC group
$\Delta=\langle \alpha, \beta_{1},\ldots,\beta_{l}: \alpha^{2}\beta_{1}\cdots \beta_{l}=\beta_{1}^{2n}=\cdots=\beta_{l}^{2n}=1\rangle.$
The quotient Klein surface uniformized by $\Delta$ is the orbifold whose underlying surface is the real projective plane and its conical points are $l$ points, each one with order $2n$. Let us consider the surjective homomorphism
$\Theta:\Delta \to Q_{2^{n}}: 
\Theta(\alpha)=y, \; \Theta(\beta_{j})=x, j=1,\ldots,l.$
If $\Gamma$ is the kernel of $\Theta$, then $\Gamma$ is a torsion-free subgroup contained in the half-orientation part $\Delta^{+}$ of $\Delta$. The closed Riemann surface $S={\mathbb H}^{2}/\Gamma$ is a closed Riemann surface admitting the group $G_{n}$ as a subgroup of ${\rm Aut}(S)$ with $\langle x \rangle$ a subgroup of ${\rm Aut}^{+}(S)$. As the signature of $S/\langle x \rangle$ is of the form $(0;2n,\ldots,2n)$, where the number of cone points is exactly $2l=2n(2q-1) \geq 4(2q-1)>6$, it follows from Singerman list of maximal Fuchsian groups \cite{Singerman} that we may choose $\Delta$ so ${\rm Aut}(S)=G_{n}$. In this case, as the only anticonformal automorphisms of $S$ are the elements of $G_{n}-\langle x \rangle$ (which have order four as seen in Section \ref{Gn}) it follows that $S$ is pseudo-real. By Riemann-Hurwitz formula, applied to the branched cover $S \to S/\langle x\rangle$, we observe that $S$ has genus $(l-1)(2n-1)$.


\end{document}